\documentclass[11pt, a4paper]{amsart}
\usepackage{amssymb}
\usepackage{amsmath}
\usepackage{amssymb}
\usepackage{amsthm}
\usepackage{bbm}
\usepackage{bm}
\usepackage{mathtools}
\usepackage{mathrsfs}
\usepackage[T1]{fontenc}
\usepackage[usenames,dvipsnames,svgnames,table]{xcolor}
\usepackage{esint}
\allowdisplaybreaks

\usepackage{enumitem}
\setenumerate{label={\rm (\alph{*})}}
\usepackage[colorlinks=true,citecolor=BrickRed,urlcolor=black,
linkcolor=NavyBlue]{hyperref}

\usepackage{graphicx}
\usepackage{tikz}

\usepackage{geometry}
\makeatletter
\newcommand*\bigcdot{\mathpalette\bigcdot@{.5}}
\newcommand*\bigcdot@[2]{\mathbin{\vcenter{\hbox{\scalebox{#2}{$\m@th#1\bullet$}}}}}
\makeatother

\numberwithin{equation}{section}

\newtheorem{theorem}{Theorem}[section]
\newtheorem{lemma}[theorem]{Lemma}
\newtheorem{proposition}[theorem]{Proposition}

\newtheorem{definition}[theorem]{Definition}

\newtheorem{remark}[theorem]{Remark}
\newtheorem{example}[theorem]{Example}

\usepackage{booktabs}

\normalsize

\def\XXint#1#2#3{{\setbox0=\hbox{$#1{#2#3}{\int}$ }
		\vcenter{\hbox{$#2#3$ }}\kern-.6\wd0}}

\newcommand{\bv}{\operatorname{BV}}
\newcommand{\di}{\operatorname{div}}
\newcommand{\dif}{\operatorname{d}\!}

\newcommand{\tr}{\operatorname{Tr}}
\newcommand{\R}{\mathbb{R}}
\newcommand{\B}{\mathbb{B}}
\newcommand{\C}{\mathbb{C}}

\newcommand{\locc}{\operatorname{loc}}

\newcommand{\wstar}{\stackrel{*}{\rightharpoonup}}

\newcommand{\sobo}{\operatorname{W}}

\newcommand{\lebe}{\operatorname{L}}
\newcommand{\hold}{\operatorname{C}}
\newcommand{\curl}{\operatorname{curl}}
\renewcommand{\leq}{\leqslant}
\newcommand{\imag}{\operatorname{i}}

\newcommand{\lin}{\operatorname{Lin}}

\usepackage{tikz-cd}
\AtEndDocument{\bigskip{\footnotesize%
		\textsc{University of Warwick, Zeeman Building, Coventry, CV4 7HP, United Kingdom} \par  
		\noindent\textit{E-mail address}, B.~Rai\c{t}\u{a}: \texttt{bogdan.raita@warwick.ac.uk}\par
		\noindent\textit{E-mail address}, A.~Skorobogatova: \texttt{a.skorobogatova@warwick.ac.uk}
}}
\begin{document}
	\title[Continuity from canceling operators]{Continuity and canceling operators of order $n$ on $\R^n$}
	\author[B. Rai\c{t}\u{a}]{Bogdan Rai\c{t}\u{a}}
	\author[A. Skorobogatova]{Anna Skorobogatova}
	\subjclass[2010]{Primary: 26D10; Secondary: 46E35 }
	\keywords{Convolution operators, Canceling operators, Critical embeddings, Sobolev inequalities, Linear $\lebe^1$-estimates, Elliptic systems, Functions of bounded variation.}
	
	\begin{abstract}
		We prove that for elliptic and canceling linear differential operators $\B$ of order $n$ on $\R^n$, continuity of a map $u$ can be inferred from the fact that $\B u$ is a measure. We also prove strict continuity of the embedding of the space $\bv^\B(\R^n)$ of functions of bounded $\B$-variation into the space of continuous functions vanishing at infinity.
	\end{abstract}
	\maketitle
	\section{Introduction}
	It was recently proved by {Van Schaftingen} in \cite{VS}, building on, among others, \cite{BB04,BB07,BBM3,BVS,VS-1,VS1,VS0}, that the linear Sobolev-type $\lebe^1$-estimate 
	\begin{align}\label{eq:VS_j}
	\|D^{k-j}u\|_{\lebe^\frac{n}{n-j}}\leq c\|\B u\|_{\lebe^1}\quad\text{ for }u\in\hold^\infty_c(\R^n,V)\quad\left [j=1\ldots\min \{k,n-1\}\right]
	\end{align}
	holds if and only if the $k$-th order elliptic partial differential operator $\B$ satisfies a then newly introduced \emph{canceling} condition (see Section~\ref{sec:prel} for notation and terminology). Of course, all estimates in \eqref{eq:VS_j} would follow from the estimate for $j=0$ by an iteration of the Sobolev inequality. However, it is well-known that linear Calder\'on-Zygmund theory fails in $\lebe^1$ in general. In particular, estimates of the type \eqref{eq:VS_j} with $j=0$ can only hold in trivial instances \cite{Ornstein,KK}. This paper is concerned with the other limiting case of \eqref{eq:VS_j}, namely $j=n$, provided that $k\geq n$. We make the convention $n/0=\infty$, and will in fact assume that $k=n$ for simplicity.
	
	Henceforth, we assume that $\B$ is an $n$-th order elliptic operator on $\R^n$. In this case, the estimate \eqref{eq:VS_j} with $j=n$ cannot be inferred from $\eqref{eq:VS_j}$ with $j=n-1$, since $\dot{\sobo}{^{1,n}}\not\hookrightarrow\lebe^\infty$. However, it is not difficult to prove that for $\B=D^n$, we have
	\begin{align}\label{eq:Linfty}
	\|u\|_{\lebe^\infty}\leq c\|\B u\|_{\lebe^1}\quad\text{ for } u\in\hold^\infty_c(\R^n).
	\end{align}
	In more generality, it was proved by {Bousquet} and {Van Schaftingen} in \cite[Thm.~1.3]{BVS} that the estimate \eqref{eq:Linfty} holds for canceling operators, while simple one-dimensional examples show that the canceling condition might not be necessary. Indeed, it was shown in \cite[Thm.~1.3]{Rdiff} that the inequality \eqref{eq:Linfty} is equivalent to a \emph{weakly canceling} condition. We briefly describe the difference between the two conditions from an analytic perspective: On one hand, $\B$ is canceling if and only if the equation $\B u=\delta_0 w$ for vectors $w$ has solutions only for $w=0$ \cite[Lem.~2.5]{Rdiff}. On the other hand, $\B$ is weakly canceling if and only if solutions of $\B u=\delta_0 w$ are locally bounded \cite[Lem.~4.3]{Rdiff}.
	
	Using the estimates \eqref{eq:VS_j} and \eqref{eq:Linfty}, one can investigate more general partial differential equations, e.g., variational problems of linear growth \cite{ARDPR,BDG}. From the point of view of weak formulations (e.g., of the Euler-Lagrange equations), it is natural to investigate the spaces
	\begin{align}\label{eq:def_sob}
	\begin{split}	\sobo^{\B,1}(\Omega)&\coloneqq\{u\in\sobo^{n-1,1}(\Omega,V)\colon \B u\in\lebe^1(\Omega,W)\},\\
	\bv^{\B}(\Omega)&\coloneqq\{u\in\sobo^{n-1,1}(\Omega,V)\colon \B u\in\mathcal{M}(\Omega,W)\}.
	\end{split}
	\end{align}
	We equip the space $\bv^\B$ (hence also $\sobo^{\B,1}$) with the complete norm
	\begin{align}\label{eq:def_norm}
	\|u\|_{\bv^\B(\Omega)} \coloneqq |\B u|(\Omega) + \sum_{|\alpha| < n}\|\partial^{\alpha}u\|_{\lebe^1(\Omega,V)}.
	\end{align}
	The emergence of the $\bv^\B$-space is natural since, from a typical $\lebe^1$-bound on a minimizing/approximating sequence, one cannot infer weak $\lebe^1$-compactness, but can infer weakly-$*$ compactness in the space of measures. Motivated by these facts, several analytical properties of $\sobo^{\B,1}$ and $\bv^\B$ were already studied in \cite{BDG,GR,Rdiff}.
	
	Starting from the estimate \eqref{eq:Linfty}, one can infer by norm-convergence of mollifications that the inclusion $\sobo^{\B,1}(\R^n)\subset\lebe^\infty(\R^n,V)$ holds (Lemma~\ref{lem:WB1_density}), is norm-continuous (Lemma~\ref{lem:WB1_ctnity}), and, more importantly, directly implies that $\sobo^{\B,1}(\R^n)\subset\hold(\R^n,V)$. The situation is quite different for $\bv^\B(\R^n)$. Indeed, we still have that \eqref{eq:Linfty} implies $\bv^\B(\R^n)\subset\lebe^\infty(\R^n,V)$, which follows by \emph{strict-density} of mollifications, convergence that is not continuous with respect to addition of measures. Interestingly, this pathology turns out to be phenomenological: one can show that cancellation is necessary for continuity of $\bv^\B$-maps to hold (Proposition~\ref{prop:nec_canc}). The purpose of of the present paper is to investigate the sufficiency of cancellation for continuity of maps in $\bv^\B(\R^n)$ (recall that we assumed that $\B$ is elliptic of order $n$).
	
	We begin with a simple example: $\sobo^{n,1}(\R^n)$ versus $\bv^n(\R^n)$, i.e., $\B=D^n$, which is elliptic for all $n\geq 1$, weakly canceling but not canceling if $n=1$ and canceling otherwise. It is well-known that in this limiting case, we have that $\sobo^{n,1}(\R^n)\hookrightarrow\hold_0(\R^n)$. In the case of $\bv$, we have that $\bv(\R)\subset\lebe^\infty(\R)$, but indicator functions of bounded intervals clearly lie in $\bv$. On the other hand, if $n\geq2$, it was shown by {Tartar} and revisited in \cite{PVS} that maps in $\bv^n(\R^n)$ are continuous (see also \cite[Thm.~1(iii)]{Dorronsoro}).
	
	Another relevant example is given by $\B=\Delta\circ(\di, \curl)$ on $\R^3$, which is elliptic, weakly canceling, but not canceling. In this case:
	\begin{enumerate}
		\item $\sobo^{\B,1}(\R^3)\hookrightarrow\hold_0(\R^3,\R^3)$ with norm-continuity,
		\item $\bv^{\B}(\R^3)\subset\lebe^\infty(\R^3,\R^3)$,
		\item $\bv^{\B}(\R^3)\not\subset\hold(\R^3,\R^3)$.
	\end{enumerate}
	The assertions above follow from \eqref{eq:Linfty}, the results in Section~\ref{sec:prel_strict}, and Proposition~\ref{prop:nec_canc}. In more generality, we have:
	\begin{theorem}\label{thm:main}
		Let $\B$ be an $n$-th order elliptic operator on $\R^n$. Then maps in $\bv^\B_{\locc}(\R^n)$ are continuous if and only if $\B$ is canceling.
	\end{theorem}
	In fact, it is not too difficult to prove that the inclusion $\bv^\B_{\locc}(\R^n)\subset \hold(\R^n,V)$ is bounded (in the sense of linear operators) if we consider the locally-uniform topology on $\hold(\R^n, V)$ (see Proposition~\ref{prop:GR}). However, as alluded to already, the natural topology to use on $\bv^\B$ is the strict topology, which we will now define.	In accordance with the better studied case of $\bv$ \cite{AFP_book}, we say that $u_j\rightarrow u$ \emph{$\B$-strictly} if and only if $u_j\rightarrow u$ in $\sobo{^{n-1,1}}$  and $\|\B u_j\|_{\mathcal{M}}\rightarrow \|\B u\|_\mathcal{M}$.
	
	Our second main result is to show that, indeed, the inclusion of Theorem~\ref{thm:main} is continuous with respect to strict convergence. Below, we write $\hold_0(\R^n)$ for the space of continuous functions vanishing at infinity which is a Banach space with the supremum norm.
	\begin{theorem}\label{thm:strict}
		Let $\B$ an $n$-th order, elliptic, and canceling operator on $\R^n$. Then the inclusion $\bv^\B(\R^n)\subset \hold_0(\R^n,V)$ is $\B$-strictly continuous.
	\end{theorem}
	In fact, we can say slightly more: even if we only have $\bv^\B\subset\lebe^\infty$ (i.e., $\B$ is weakly canceling), we still have that the inclusion is locally strictly continuous if and only if $\B$ is canceling (see Lemma~\ref{lem:nec_strict}).
	
	The proof of strict continuity is realized in two steps. Firstly, we use the concentration compactness principle of \cite{PL^2} to prove that the intermediate embeddings $\bv^\B(\R^n)\subset\sobo{^{n-j,n/(n-j)}}$ for integers $0<j<n$ are strictly continuous.
	Secondly, we employ a fairly elementary measure theoretic argument, using the Arzela-Ascoli theorem and an estimate on the modulus of continuity of $\bv^\B$ maps obtained in the proof of Theorem~\ref{thm:main}. This is in contrast with with the approach of the upcoming paper \cite{GRVS}, where strict continuity of embeddings of $\bv^\B(\R^n)$ on lower dimensional subsets of $\R^n$ is discussed. There, the approach to strict continuity relies on proving a multiplicative trace inequality for elliptic and canceling operators in the spirit of \cite[Sec.~1.4.7]{Mazya}. In our case, it does not seem possible to employ an interpolation inequality, particularly since, for instance, the embedding $\dot{\sobo}{^{s,n/s}}(\R^n)\hookrightarrow\lebe^\infty$ holds only for $s=n$ in the range $s\leq n$. For more detail on related multiplicative inequalities in Sobolev spaces, see the recent the work \cite{BM} and the references therein.
	
	Of course, possible applications may require variants of Theorem~\ref{thm:main} on bounded domains. Theorem~\ref{thm:main} only guarantees that for canceling operators $\B$, we have that $\bv^\B(\Omega)\subset\hold(\Omega,V)$ whenever $\Omega\subset\R^n$ is open. However, this statement does not, in any way, imply continuity up to the boundary. In particular, it is not too difficult to infer from the results in \cite{GR}, which concern integrability near the boundary, that for the inclusion $\bv^\B(\Omega)\subset\lebe^\infty$ to hold, the operator $\B$ needs to be $\C$-elliptic, i.e., $\ker_\C\B(\xi)=\{0\}$ for any $0\neq\xi\in\C^n$ (see Lemma~\ref{lem:nec_C-ell}). It was also shown in \cite{GR} that the $\C$-ellipticity condition is \emph{strictly} stronger than cancellation in the class of elliptic operators. Using our method for proving Theorem~\ref{thm:main}, one can  prove the following:
	\begin{theorem}\label{thm:main2}
		Let $\B$ be an $n$-th order elliptic operator on $\R^n$ and $Q\subset\R^n$ be an open cube. Then  $\bv^\B(Q)\subset\hold(\bar Q,V)$ if and only if $\B$ is $\C$-elliptic. Moreover, if $\B$ is $\C$-elliptic, then there exists a bounded, linear trace operator $$\tr\colon\bv^\B(Q)\rightarrow\hold(\partial Q,V).$$
	\end{theorem}
	More general domains are also possible, but we chose to present Theorem~\ref{thm:main2} in this form for better comparison with the only other similar result (apart from $\bv^n$) that we are aware of, namely \cite[Thm.~1.4]{PVS}. There, it is shown that if one takes $\B=\partial_1\partial_2\ldots\partial_n$, the statement of the Theorem~\ref{thm:main} holds for $Q$ chosen as the unit cube. This shows that ellipticity need not be necessary for a result of this type; however, it is not clear whether one can cover more than special domains for non-elliptic operators (see also the remarks following \cite[Thm.~1.5]{PVS}).
	
	A similar theme between the mechanisms used to prove continuity in \cite{PVS} and our results is that lack of concentration of $\B u$ on a class $\mathcal{S}$ of sets implies continuity of $u$. In the case of \cite{PVS}, $\mathcal{S}$ consists of the hyperplanes of co-dimension $1$ orthogonal to the coordinate axes, whereas in our case, $\mathcal{S}$ is the class of singletons. In fact, our approach generalizes the sketch in the case $\B=D^n$ given after \cite[Thm.~1.2]{PVS}.
	
	This paper is organized as follows: In Section~\ref{sec:prel} we collect the basic notation and background results on partial differential operators and strict-density and strict-continuity in $\bv^\B$. In Section~\ref{sec:canc} we prove Theorems~\ref{thm:main} and \ref{thm:strict} which concern canceling operators. In Section~\ref{sec:C-ell} we prove the boundary continuity Theorem~\ref{thm:main2} for $\C$-elliptic operators.
	{\subsection*{Acknowledgement} The authors are thankful to J. Kristensen for suggesting the problem and to F. Gmeinder and F. Rindler for helpful discussion around the topic of the paper.
	This project has received funding
	from the European Research Council (ERC) under the European Union's Horizon	2020 research and innovation programme under grant agreement No 757254 (SINGULARITY).}
	
	\section{Preliminaries}\label{sec:prel}
	\subsection{Basic notation}
	Throughout this paper, the underlying space for all functions and measures is denoted by $\R^n$. For most purposes, one can assume that $n>1$. We will write $\Omega\subset\R^n$ for a typical domain, which will often be assumed open, bounded, and having Lipschitz boundary.
	
	We will use standard notation for the $\lebe^p$ and Sobolev spaces, $\sobo^{k,p}$. We will also use the homogeneous Sobolev space $\dot{\sobo}{^{k,p}}(\R^n)$, which is defined as the closure of $\hold_c^\infty(\R^n)$ in the (semi-)norm $u\mapsto\|D^k u\|_{\lebe^p(\R^n)}$. Here, $\hold^\infty_c$ denotes the space of compactly supported, smooth functions. We will also work with the space $\hold(\Omega)$ of continuous functions, which is naturally endowed with the topology of locally uniform convergence. As a particularly important subspace, we single out $\hold_0(\R^n)$, the space of continuous functions vanishing at infinity, which is the uniform closure of $\hold^\infty_c(\R^n)$, hence a Banach space. We will denote by $\mathscr{S}(\R^n)$ the Schwarz space of rapidly decreasing functions, by $\mathscr{S}^\prime(\R^n)$ its linear topological dual, the space of tempered distributions, and by $\mathscr{D}^\prime(\R^n)$ the space of distributions.
	
	The Lebesgue-$n$ and Hausdorff-$\alpha$ measures will be denoted by $\mathscr{L}^n$ and $\mathscr{H}^\alpha$. In more generality, we will work with vectorial finite measures $\mu\in\mathcal{M}(\Omega,W)$, where $W$ is a finite dimensional normed vector space. Here, finite means that the total variation norm $\|\mu\|_{\mathcal{M}(\Omega)}=|\mu|(\Omega)$ is finite, where $|\mu|$ denotes the total variation measure of $\mu$. We also briefly recall the Riesz representation theorem, which states that $\mathcal{M}(\Omega)$ is the linear dual of $\hold_0(\Omega)$. This identification enables us to characterize weakly-* convergence of measures. For more details on elementary measure theory, we refer the reader to \cite{AFP_book}.
	
	We write $B(x,r)=B_r(x)\subset\R^n$ for the ball of radius $r$, centered at $x$, and by $Q(x,r)=Q_r(x)$ for the cube of side length $2r$, centered at $x$. If $x$ lies in an ambient domain $\Omega$, the notation $B_r(x)$ will actually denote $B_r(x)\cap\Omega$, with an analogous convention for $Q_r(x)$. We will denote the average of an integrable function $f$ on $\Omega$ (always taken with respect to $\mathscr{L}^n$) by $(f)_\Omega=\fint_\Omega f\dif x=\mathscr{L}{^n}(\Omega){^{-1}}\int_\Omega f\dif x$. In the case when $\Omega$ itself is a ball or cube of radius $r$ around $x$, we will indiscriminately use the notation $(f)_{x,r}$.
	
	We use the notation $c$ to denote a general constant $c>0$ which does not depend on any of the quantities that vary in a line of estimation. The value of $c$ can, of course, vary from line to line.
	\subsection{Partial differential operators}\label{sec:prel_pdos}
		We will consider $k$-homogeneous linear differential operators $\B$ on $\R^n$ from $V$ to $W$, with constant coefficients:
	\begin{align}\label{eq:B}
	\B u\coloneqq \sum_{|\beta|=k}B_\beta \partial^\beta u\quad\text{ for }u\in\mathscr{D}^\prime (\R^n,V).
	\end{align}
	Here $B_\beta\in\lin(V,W)$ for all multi-indices $\beta\in\mathbb{N}^n_0$ with $|\beta|=n$ and $V,\,W$ are finite dimensional inner product spaces. We also consider the symbol map 
	\begin{align*}
	\B (\xi)\coloneqq \sum_{|\beta|=k}\xi^\beta B_\beta\in\lin(V,W)  \quad\text{ for }\xi\in\R^n,
	\end{align*}
	which is defined such that $\widehat{\B u}(\xi)=(-\imag)^k\B(\xi)\hat u(\xi)$ for $\xi\in\R^n$, $u\in\mathscr{S}(\R^n,V)$. Our convention for the Fourier transform is
	\begin{align*}
	[\mathscr{F}u](\xi)=\hat{u}(\xi)\coloneqq\int_{\R^n}u(x)\mathrm{e}^{-\imag x\cdot\xi}\dif x,
	\end{align*}
	defined for $u\in\mathscr{S}(\R^n)$ and extended by duality to tempered distributions.
	
	We say that an operator $\B$ is \emph{elliptic} if and only if $\ker_\R\B(\xi)=\{0\}$ for all $0\neq\xi\in\R^n$. Under this assumption, one has that the tensor-valued $(-k)$-homogeneous function defined by $\B^\dagger(\xi)\coloneqq\left[\B^*(\xi)\B(\xi)\right]^{-1}\B^*(\xi)$ for $\xi\neq0$ and $\B^\dagger(0)=0$ is a left-inverse of $\B(\xi)$. In particular, one has for $u\in\mathscr{S}(\R^n,V)$ and $j=0,\ldots ,k$ that
	\begin{align*}
	\widehat{D^{k-j}u}(\xi)=\imag^{k-j}\B^\dagger(\xi)\widehat{\B u}(\xi)\otimes\xi^{\otimes(k-j)}\eqqcolon m_{-j}(\xi)\widehat{\B u}(\xi)\quad\text{ for }0\neq\xi\in\R^n,
	\end{align*}
	where $m_{-j}\in\hold^\infty(\R^n\setminus\{0\},\imag^{k-j}\lin(W,V\odot^{k-j}\R^n))$ is a $(-j)$-homogeneous map. Here $V\odot^l\R^n$ denotes the space of $V$-valued, symmetric $l$-linear maps on $\R^n$. In particular, $m_{-j}$ is a tempered distribution with reasonably well-understood inverse Fourier transform (see \cite[Lem.~2.1]{BVS}, \cite[Eq.~(3,1)]{Rdiff}, and the references therein):
	\begin{lemma}\label{lem:conv}
		Let $\B$ be an elliptic operator as in \eqref{eq:B}. Then there exist convolution kernels $K_j\in\hold^\infty(\R^n\setminus\{0\},\lin(W,V\odot^{k-j}\R^n))$, $j=0,\ldots,\min\{k,n\}$, such that for all $u\in\hold^\infty_c(\R^n,V)$ we have that $D^{k-j}u=K_j\star \B u$, where
		\begin{align*}
		K_j=H_{j-n}\quad&\text{ if }j=0,\ldots,\min\{k,n-1\},\\
		K_n=H_0+\log|\cdot|\mathbf{L}\quad&\text{ if }j=n\leq k.
		\end{align*}
		Here $H_{l}$ is $l$-homogeneous and $\mathbf{L}\in\lin(W,V\odot^{k-n}\R^n)$ is defined by
		\begin{align}\label{eq:L_wcanc}
		\mathbf{L}w\coloneqq\int_{\mathbb{S}^{n-1}}\B^\dagger(\xi)w\otimes\xi^{\otimes(k-n)}\dif\mathscr{H}^{n-1}(\xi)\quad\text{ for }w\in W.
		\end{align}
	\end{lemma}
	We further recall that an operator $\B$ is said to be \emph{canceling} if and only if 
	\begin{align}\label{eq:intersection}
	\mathcal{I}\coloneqq\bigcap_{\xi\in\mathbb{S}^{n-1}}\mathrm{im\,}\B(\xi)=\{0\}.
	\end{align}
	It is essentially shown in \cite[Lem.~2.5]{Rdiff} based on \cite[Prop.~6.1]{VS} that, for an elliptic operator $\B$, we have that $\mathcal{I}=\{w\in W\colon  \B u=\delta_0w \text{ for some }u\in\lebe^1_{\locc}(\R^n,V)\}$. In particular, in the class of elliptic operators, the canceling operators are precisely those for which the space of measures $\{\B u\}$ contains no Dirac measures. In contrast, an operator is said to be \emph{weakly canceling}, as introduced in \cite[Sec.~1.3]{Rdiff}, if and only if
	$\mathbf{L}(\mathcal{I})=\{0\}$. In view of Lemma~\ref{lem:conv}, it is not difficult to see that an elliptic operator $\B$ is weakly canceling if and only if $D^{k-n}u\in\lebe^\infty_{\locc}$ whenever $\B u=\delta_0 w$ for some $w\in W$ \cite[Lem.~4.3]{Rdiff}. We have the following:
	\begin{lemma}\label{lem:char(w)canc}
		Let $\B$ be an elliptic operator as in \eqref{eq:B}, $\mathcal{I}$, $\mathbf{L}$ be defined by \eqref{eq:intersection}, \eqref{eq:L_wcanc} respectively, and consider the system
		\begin{align}\label{eq:canc_eq}
		\B u=\delta_0 w\quad\text{ for }u\in\lebe^1_{\locc}(\R^n,V),\,w\in W.
		\end{align}
		Then:
		\begin{enumerate}
			\item $\B$ is weakly canceling (i.e., $\mathbf{L}(\mathcal{I})=\{0\}$) if and only if all solutions $u$ of \eqref{eq:canc_eq} are such that $D^{k-n}u\in\lebe^\infty_{\locc}$.
			\item $\B$ is canceling (i.e., $\mathcal{I}=\{0\}$) if and only if \eqref{eq:canc_eq} implies that $w=0$.
		\end{enumerate}
	\end{lemma}
	Finally, we remark that cancellation is indeed strictly stronger than weak cancellation. This can be seen already from the Introduction, where the operator $\Delta\circ(\di,\curl)$ on $\R^3$ was mentioned. More examples illustrating the difference between the two classes can be found in \cite[Sec.~4.3]{Rdiff}.
	
	We will also discuss $\C$-elliptic operators, originally introduced by Smith in \cite{Smith_old,Smith}, aiming to obtain boundary $\lebe^p$-estimates for linear systems in non-smooth (e.g., Lipschitz) domains. Such operators were later studied in \cite{Kala1,Kala2} and, more recently, in connection with boundary $\lebe^1$-estimates in \cite{BDG,GR}. In fact, the terminology was first used in \cite{BDG}, for first order operators. Apart from estimates near and on the boundary, $\C$-elliptic operators are often relevant in applications, e.g., plasticity \cite{AnGi,StTe}, fracture mechanics \cite{ChCr}, and image reconstruction \cite{DaFoLi}.
	
	Recall that an operator $\B$ as in \eqref{eq:B} is said to be \emph{$\C$-elliptic} if and only if $\ker_\C\B(\xi)=\{0\}$ for all $0\neq\xi\in\C^n$. Of course, $\C$-elliptic operators are elliptic. The converse is not true, as can be seen from the example $\B=(\di,\curl)$. 
	
	We will use the following embedding result, which was not mentioned in \cite{GR}, but easily follows from \cite[Thm.~1.2, Lem.~3.2]{GR} and the strict density lemma \cite[Lem.~4.15]{BDG} (the latter concerns first order operators, but is easily adapted to higher order operators, as can be seen from the upcoming proof in \cite{GRVS}).
	\begin{proposition}\label{prop:GR}
		Let $n>1$, $\B$ as in \eqref{eq:B} be an $n$-th order operator, and $\Omega\subset\R^n$ be a bounded Lipschitz domain. Then $\B$ is $\C$-elliptic if and only if
		\begin{align*}
		\|u\|_{\lebe^\infty(\Omega,V)}\leq c_\Omega\left(|\B u|(\Omega)+\|u\|_{\lebe^1(\Omega,V)}\right)\quad\text{ for }u\in\bv^\B(\Omega).
		\end{align*}
	\end{proposition}
		Here $c_\Omega>0$ depends both on the size and the geometry of $\Omega$.
	\begin{proof}
		Suppose that $\B$ is $\C$-elliptic. We know from \cite[Thm.~1.2]{GR} and \ref{eq:Linfty} that the claimed inequality holds for maps $u\in\hold^\infty(\bar\Omega,V)$. Let $u\in\bv^\B(\Omega)$. One can then modify \cite[Lem.~4.15]{BDG} to show there exist $u_j=T_ju\in \hold(\bar\Omega,V)$ that converge strictly to $u$ in $\bv^\B(\Omega)$. By passing to a subsequence, we can assume that $u_j$ converges to $u$ $\mathscr{L}^n$-a.e., so that
		\begin{align*}
			|u(x)|&=\lim_j |u_j(x)|\leq c_\Omega \lim_j\left(\|\B u_j\|_{\lebe^1(\Omega)}+\| u_j\|_{\lebe^1(\Omega)}\right)\\&=c_\Omega \left(\|\B u\|_{\lebe^1(\Omega)}+\| u\|_{\lebe^1(\Omega)}\right)
		\end{align*}
		for $\mathscr L^n$-a.e. $x\in\Omega$. This concludes the proof of sufficiency.
		
		Necessity of $\C$-ellipticity follows as in the final part of \cite[Sec.~4.2]{GR} (alternatively, see the proof of Lemma~\ref{lem:nec_C-ell}, where we construct a map in $\bv^\B(\Omega)$ that is unbounded near the boundary).
	\end{proof}
	We conclude this section by giving a class of (academic) examples of $\C$-elliptic operators of order $2$ on $\R^2$ that do not seem to reduce to the operator $\partial_1\partial_2$ which was analyzed in \cite{PVS}.
	\begin{example}
		Let $a,b>0$ be two distinct real numbers. Then the operator
		\begin{align*}
			\B(\xi)\coloneqq \left(\xi_1^2+a\xi_2^2,\,\xi_1^2+b\xi_2^2\right),\quad\text{ for }\xi\in\R^2,
		\end{align*}
		is $\C$-elliptic on $\R^2$ from $\R$ to $\R^2$.
	\end{example}
	In general, one can consider two homogeneous real polynomials $p_1,\,p_2$ on $\R^2$ that have no common non-trivial complex roots and set $\B(\xi)\coloneqq(p_1(\xi),\,p_2(\xi))$ for $\xi\in\R^2$.
	\subsection{Strict density and continuity}\label{sec:prel_strict}
	We extend the definition of $\sobo^{\B,1}$ and $\bv^\B$ from \eqref{eq:def_sob}, \eqref{eq:def_norm} to operators of order $k$ in an obvious manner. 
	
	We begin by discussing norm-density of smooth maps in $\sobo^{\B,1}$, as well as norm-continuity of the embeddings of $\sobo^{\B,1}$.
	\begin{lemma}\label{lem:WB1_density}
		Let $\B$ be an operator as in \eqref{eq:B} and $u\in\lebe^1_{\locc}(\R^n,V)$ be such that $\B u\in\lebe^1(\R^n,W)$. Then there exist $u_j\in\hold^\infty_c(\R^n,V)$ be such that $\B u_j\rightarrow \B u$ in $\lebe^1(\R^n,W)$.
	\end{lemma}
	
	\begin{proof}
		Fix $N>0$, choose $\rho_N \in \hold_c^\infty(B_N(0), V)$ and let $u^N = \rho_N u$, so that $u^N\in \sobo^{\B,1}_c(B_N(0), V)$. Applying a standard mollification argument with a partition of unity (see \cite[Sec.~1.1.5]{Mazya}), we obtain a sequence $\{u^N_j\} \subset \hold^\infty_c(B_N(0), V)$ with $\B u^N_j\rightarrow \B u^N$ in $\lebe^1(B_N(0),W)$, where the compact support follows from the compact essential support of $u^N$ in $B_N(0)$. Extending by zero to $\R^n$ and combining this with the fact that $\B u^N \to \B u$ in $\lebe^1(\R^n,W)$, the result follows when we extract a diagonal subsequence.
	\end{proof}
	It follows that in $\sobo^{\B,1}$ we have both norm-density, as well as norm-continuity:
	\begin{lemma}\label{lem:WB1_ctnity}
		Let $j=1,\ldots, \min\{k,n\}$ and $\B$ as in \eqref{eq:B} be elliptic and
		\begin{enumerate}
			\item canceling if $j<n$;
			\item weakly canceling if $j=n\leq k$.
		\end{enumerate}
		Then the embedding $\sobo^{\B,1}(\R^n)\hookrightarrow\dot{\sobo}{^{k-j,n/(n-j)}}(\R^n,V)$ is norm-continuous.
	\end{lemma}
	\begin{proof}
		Let $u \in \sobo^{\B,1}(\R^n)$ and find a sequence $\{u_j\} \subset \hold^\infty_c(\R^n,V)$ as given by Lemma \ref{lem:WB1_density} that also converges to $u$ in $\sobo^{k-1,1}(\R^n,V)$. The conclusion follows from applying the estimates \eqref{eq:VS_j} and \eqref{eq:Linfty} to show that $\{u_j\}$ is Cauchy in $\dot{\sobo}{^{k-j,n/(n-j)}}$. The $\lebe^1$-convergence of $u_j$ to $u$ allows us to verify that the limit is the appropriate derivative of $u$ by the uniqueness of weak derivatives. It follows that the estimates \eqref{eq:VS_j} or \eqref{eq:Linfty} hold as appropriate for maps in $\sobo^{\B,1}(\R^n)$.
		
		Now assume that $v_j\rightarrow v$ in $\sobo^{\B,1}(\R^n)$. The fact that $D^{k-j}v_j\rightarrow D^{k-j}v$ in $\lebe^{n/(n-j)}$ follows directly from \eqref{eq:VS_j} or \eqref{eq:Linfty}.
	\end{proof}
	In particular, by taking $j=k=n$, we immediately see that weak cancellation suffices to guarantee continuity of $\sobo^{\B,1}$-maps when $\B$ has order $n$:
	\begin{proposition}
		Let $\B$ as in \eqref{eq:B} be elliptic and weakly canceling of order $n$. Then maps in $\sobo^{\B,1}(\R^n)$ are continuous (vanishing at infinity).
	\end{proposition}
	
	This is in sharp contrast with the situation in $\bv^\B$, as illustrated by Lemma~\ref{prop:nec_canc}. On a different note, we also cannot expect norm-density results as for $\sobo^{\B,1}$, simply since $\lebe^1$-limits of $\hold^\infty_c$-sequences are $\lebe^1$-maps. In turn, it is easy to see that mollifications of measures converge weakly-* and, moreover, we have that:
	\begin{lemma}
		Let $\mu\in\mathcal{M}(\R^n)$. Then $\int_{\R^n}|\mu\star \eta_\varepsilon|\dif x\rightarrow |\mu|(\R^n)$ as $\varepsilon\downarrow0$. Here $\eta_\varepsilon$ is a standard sequence of mollifiers.
	\end{lemma}
	In general, we say that if $\mu_j\wstar\mu$ and $|\mu_j|(\R^n)\rightarrow|\mu|(\R^n)$, then the sequence of measures $\mu_j$ is said to converge \emph{strictly} to $\mu$, which is {consistent with definition of $\B$-strict convergence from the Introduction}, which we recall here:
	\begin{definition}[Strict convergence]\label{def:strict}
		We say that $(u_j)_j \subset \bv^\B(\Omega)$ converges to $u$ \emph{$\B$-strictly} (or \emph{strictly in $\bv^\B(\Omega)$}, or just \emph{strictly}) if
		\[|\B u_j|(\Omega)\rightarrow|\B u|(\Omega) \qquad \text{and} \qquad u_j \rightarrow u \quad \text{in} \ \sobo^{k-1,1}(\Omega,V).
		\]
	\end{definition}
	In particular, this definition implies that both $(\B u_j)_j$ and $(|\B u_j|)_j$ converge weakly-* to the expected limits (the former by an elementary argument, and the latter by Reshetnyak's continuity theorem \cite[Thm.~2.39]{AFP_book}). Do note that the strong convergence of the lower derivatives prevents oscillations that the convergence of the masses of $\B u_j$ need not detect.
	With this definition, we have the following strict-density result for $\bv^\B$-functions:
	\begin{lemma}\label{lem:BVB_density}
		Let $\B$ be an operator as in \eqref{eq:B} and $u\in\bv^\B(\R^n)$. Then there exist $u_j\in\hold^\infty_c(\R^n,V)$ converging $\B$-strictly to $u$.
	\end{lemma}
	\begin{proof}
		Let $u^N$ denote the same truncation of $u$ as in the proof of Lemma \ref{lem:WB1_density}, and once again combine mollification with an appropriate partition of unity (details can be found in \cite[Thm~3.9]{AFP_book}). This gives us convergence in $\sobo^{k-1,1}$ of a sequence $\{u^N_j\} \subset \hold^\infty_c(B_N(0), V)$ to $u^N$, coupled with the estimate
		\[\limsup_{j \to \infty}\int_{B_N(0)} |\B u^N_j|dx \leq |\B u^N|(B_N(0)).
		\]
		The weak-* compactness of $\lebe^1$-bounded subsets of $\mathcal{M}(\B_N(0),W)$ allows us to extract a weakly-* convergent subsequence and use lower semicontinuity of the total variation to obtain $\B$-strict convergence of $u^N_j$ to $u^N$. Combining this with the $\B$-strict convergence of $u^N$ to $u$ completes the proof.
	\end{proof}
	The strict-density Lemma above guarantees validity of the inclusions $\bv^{\B}(\R^n)\subset\dot{\sobo}{^{k-j,n/(n-j)}}(\R^n,V)$ under the assumptions of Lemma~\ref{lem:WB1_ctnity}. As for the (strict) continuity of these inclusions, we have:
	\begin{lemma}\label{lem:BVB_strict}
		Let $j=1,\ldots, \min\{k,n-1\}$ and $\B$ as in \eqref{eq:B} be elliptic and canceling. Then the embedding $\bv^\B(\R^n, V)\hookrightarrow{\sobo}^{k-j,n/(n-j)}(\R^n,V)$ is strictly-continuous.
	\end{lemma}
	This follows by an adaptation of \cite[Prop.~3.7]{RS} {originating in} the concentration compactness principle from \cite{PL^2}. Of course, from our point of view, this represents the subcritical case; one of the main results of this work is to prove that the same holds if $j=n\leq k$, which we restrict to $j=n=k$ for simplicity. As will become transparent for the reader, the ideas employed in the proof of Lemma~\ref{lem:BVB_strict} cannot be extended in the limit case of Theorem~\ref{thm:strict}.
	
	To complete the proof of Lemma~\ref{lem:BVB_strict}, in contrast to \cite[Prop.~3.7]{RS}, we also need to deal with possible concentrations at infinity. To this end, we will use a variant of Prokhorov's theorem, which can be inferred from \cite[Thm.~1.208,~Prop.~1.206]{FL_book} and the paragraph following Definition~\ref{def:strict}:
	\begin{lemma}\label{lem:prokhorov}
		Let $\mu,\,\mu_j\in\mathcal{M}(\R^n,[0,\infty))$ be such that $\mu_j\wstar\mu$ in $\mathcal{M}(\R^n)$ and $\mu_j(\R^n)\rightarrow\mu(\R^n)$. Then the sequence $(\mu_j)_j$ is \emph{tight}, i.e., for each $\varepsilon>0$, there exists a sufficiently large compact set $K\subset\R^n$ such that $\mu_j(\R^n\setminus K)\leq \varepsilon$.
		
		Let $\B$ be as in \eqref{eq:B} and suppose that $u_j\rightarrow u$ strictly in $\bv^\B(\R^n)$. Then the sequence $(|\B u_j|)_j$ is tight.
	\end{lemma}
	\begin{proof}[Proof of Lemma~\ref{lem:BVB_strict}]
		Suppose that $u_m \rightarrow u$ $\B$-strictly. We need to show that $D^{k-j}u_m \to D^{k-j}u$ in $\lebe^\frac{n}{n-j}$ for each $j=1,\ldots, \min\{k,n-1\}$. By the Vitali convergence theorem \cite[Thm.~2.24]{FL_book}, we know that the claim is equivalent with:
		\begin{enumerate}
			\item $D^{k-j}u_m\rightarrow D^{k-j}u$ in measure;
			\item\label{itm:UI} $\{D^{k-j}u_m\}_m$ is $\frac{n}{n-j}$-uniformly integrable;
			\item\label{itm:conc_infty} for every $\varepsilon>0$ there exists a Borel set $E$ with $\mathscr{L}^n(E)<\infty$ such that $$\sup_m \int_{\R^n\setminus E}|D^{k-j}u_m|^{n/(n-j)}\leq \varepsilon.$$
		\end{enumerate}
	
		Convergence in measure follows from the assumed $\lebe^1$-convergence of $(D^{k-j}u_m)_m$ and Vitali's Theorem.
		
		We next prove \ref{itm:conc_infty}. Let $\varepsilon>0$. By Lemma~\ref{lem:prokhorov}, we have that there exists large ball $B$ such that, for all $m$, $|\B u_m|(\R^n\setminus  B)\leq \varepsilon$. Let $s$ be such that $\int_{\R^n}|D^{k-l}u_m-D^{k-l}u|\dif x\leq \varepsilon$ for $m\geq s$ and $l=1,\ldots, k$. By possibly enlarging $B$, we can assume that $\int_{\R^n\setminus B}|D^{k-l}u|\dif x\leq \varepsilon$ and $\int_{\R^n\setminus B}|D^{k-l}u_m|\dif x\leq \varepsilon$ for $m<s$ and $l=1,\ldots,k$. Consider now a larger ball $\tilde B\Supset B$ and a function $\varphi \in\hold^\infty(\R^n,[0,1])$ that equals $1$ in $\tilde B$, equals $0$ in $B$ and satisfies $\|D^l\rho\|_{\lebe^\infty}\leq c$ for $l=0,\ldots k$. We estimate:
		\begin{align*}
		\|D^{k-j}u_m\|_{\lebe^{\frac{n}{n-j}}(\R^n\setminus \tilde B)}&\leq \|D^{k-j}(\varphi u_m)\|_{\lebe^{\frac{n}{n-j}}(\R^n)}\leq c|\B(\varphi u_m)|(\R^n)\\
		&\leq c\left(|\varphi\B u_m|(\R^n)+\sum_{l=1}^k \int_{\R^n}|D^l\varphi||D^{k-l}u_m|\dif x\right)\\
		&\leq c\left( |\B u_m|(\R^n\setminus B)+\sum_{l=1}^k\int_{\tilde B\setminus B}|D^{k-l}u_m|\dif x\right).
		\end{align*}
		Therefore, if $m<s$, we directly obtain that $\|D^{k-j}u_m\|_{\lebe^{\frac{n}{n-j}}(\R^n\setminus \tilde B)}\leq c\varepsilon$. If $m\geq s$, we estimate further
		\begin{align*}
		\|D^{k-j}u_m\|_{\lebe^{\frac{n}{n-j}}(\R^n\setminus \tilde B)}&\leq c\left(\varepsilon + \sum_{l=1}^k\int_{\tilde B\setminus B}|D^{k-l}u_m-D^{k-l}u|+|D^{k-l}u|\dif x\right)\leq c\varepsilon.
		\end{align*}
		The proof of \ref{itm:conc_infty} is complete.
		
		Assume now for contradiction that the claim fails; equivalently, \ref{itm:UI} fails. By strict convergence and Lemmas~\ref{lem:WB1_ctnity} and \ref{lem:BVB_density} we have that $\{D^{k-j}(u_m - u)\}_m$ is bounded in $\lebe^{\frac{n}{n-j}}(\R^n)$ and also that $\sup_m |\B u_m - \B u|(\R^n) < \infty$, so by weak-$*$ compactness, up to a subsequence (not re-labelled) we obtain
		\begin{equation}\label{w_star}|D^{k-j}(u_m - u)|^{\frac{n}{n-j}} \wstar \mu, \quad |\B u_m - \B u| \wstar \nu \quad \text{in} \ \mathcal{M}(\R^n).
		\end{equation}
		We claim that $\mu(\R^n) > 0$. By our assumption, $D^{k-j}u_m \nrightarrow D^{k-j}u$ in $\lebe^{\frac{n}{n-j}}$, so we can extract a subsequence (not re-labelled) such that
		\begin{equation}\label{not_conv}\lVert D^{k-j}(u_m - u) \rVert_{\lebe^{\frac{n}{n-j}}} \geq c
		\end{equation}
		for some $c>0$. We infer from \ref{itm:conc_infty} and $\frac{n}{n-j}$-integrability of $u$ that the sequence of measures $\{|D^{k-j}(u_m - u)|^{\frac{n}{n-j}}\}_m$ is tight, so the lower bound \eqref{not_conv} should still hold when we restrict to some sufficiently large closed ball $\bar B_R(0)$. Thus, by weakly-* upper semi-continuity on compact sets (\cite[Prop.~1.203(ii)]{FL_book}), we indeed have
		\[\mu(\bar B_R(0)) \geq c > 0.
		\]
		Now for any $\rho \in \hold^\infty_c(\R^n,[0,1])$, we also have
		\begin{align}\label{RS_trunc_est}\nonumber\|D^{k-j}(\rho u_m) \|_{\lebe^{\frac{n}{n-j}}} &\leq c |\B (\rho u_j)|(\R^n) \\
		&\leq c \left(\int \rho \dif|\B u_m| + \sum_{l=1}^k \int |D^l\rho||D^{k-l}u_m|\dif x\right).
		\end{align}
		Strict convergence combined with Reshetnyak's Continuity Theorem \cite[Thm~2.39]{AFP_book} allow us to pass to the limit on the right-hand side. To deal with the left-hand side, we raise both sides to the power of $\frac{n}{n-j}$ and use the Brezis-Lieb Lemma \cite{BL} (up to subsequence), exactly as in \cite[Prop.~3.7]{RS}. We thus have
		\begin{align}\label{RS_limiting_est}\nonumber\|D^{k-j}(\rho u) \|_{\lebe^{\frac{n}{n-j}}}^{\frac{n}{n-j}} &+ \int |\rho|^{\frac{n}{n-j}} \dif\mu \\
		&\leq c \left(\int |\rho|\dif|\B u| + \sum_{l=1}^k \int |D^l\rho||D^{k-l}u|\dif x\right)^{\frac{n}{n-j}}.
		\end{align}
		Now suppose $\mu = a_0\delta_{x_0}$, with $a_0 > 0$ and $x_0 \in \R^n$. Taking $\rho(x) = \eta\big(\frac{x-x_0}{\varepsilon}\big)$, where $\eta \in \hold^\infty_c(B_1(0),[0,1])$ with $\eta(0) = 1$, and letting $\varepsilon \searrow 0$, we get
		\[0< a_0 < |\B u|(\{x_0\})^{\frac{n}{n-j}}\]
		However, since $\B$ is elliptic and canceling, we know that $\B$ cannot charge points, so we reach a contradiction. This can be seen directly from the co-canceling estimate \cite[Thm.~1.4]{VS} (in conjunction with \cite[Prop.~4.2]{VS}), or from \cite[Thm.~3]{RW}, or from the recent paper \cite{ARDPHR}.
		
		We will now show that the positive measure $\mu\not\equiv0$ can only be a countable linear combination of Diracs, which gives us the contradiction we desire by the above argument, since at least one of these Diracs must have a positive weight attached to it. Consider the estimate \eqref{RS_trunc_est} with $u_m$ replaced by $u_m - u$. For any Borel set $E \subset \R^n$, choose $\rho$ to approximate $\mathbbm{1}_E$ and take $m \to \infty$, using the $\sobo^{k-1,1}$-convergence of $u_m$ to $u$, the Brezis-Lieb Lemma again, and \eqref{w_star}. This tells us that
		\[\frac{\mu(E)}{\nu(E)} \leq c \nu(E)^{\frac{n}{n-j}-1},
		\]
		which in turn implies, combined with the Besicovitch Derivation Theorem, that $\frac{d\mu}{d\nu} = 0$ other than on atoms on $\nu$. Thus, since $\mu \ll \nu$, we deduce that $\mu$ must be purely atomic and the result follows. Further details can be found in \cite{RS}. The contradiction reached implies that \ref{itm:UI} holds as well, hence the proof of the Lemma is complete. 
	\end{proof}

	\section{Canceling operators}\label{sec:canc}
	In this Section, we prove Theorem~\ref{thm:main}, split between Proposition~\ref{prop:nec_canc} (necessity) and Proposition~\ref{prop:EC_subset_C} (sufficiency). The second main result on elliptic and canceling operators, Theorem~\ref{thm:strict}, will be recast as Proposition~\ref{prop:strict}.
 	
We begin by noting that for $n=1$, first order elliptic operators on $\R$ are of the form $\B(\xi)=\xi M$, where $M\in\lin(V,W)$ is a matrix with $\ker M=0$. Such operators are easily seen to always be weakly canceling but never canceling.

Restricting to $n\geq2$, we first prove that cancellation is necessary for continuity:
\begin{proposition}\label{prop:nec_canc}
	Let $\B$ be an elliptic operator of order $n$ on $\R^n$. Suppose that $\bv^\B_{\locc}(\R^n)\subset\hold(\R^n,V)$. Then $\B$ is canceling.
\end{proposition}
\begin{proof}
	Let $v\in\mathscr{D}^\prime(\R^n,V)$, $w\in W$ be such that $\B v=\delta_0w$. Let $u=K_n\star\B v$ for $K_n$ as in Lemma~\ref{lem:conv}, so that $\B u=\delta_0w$ and $D^{n-j}u\in\lebe^{n/(n-j)}_\mathrm{w}$ for $j=0,\ldots,n$ (for details on these facts see \cite[Lem.~2.5,~Sec.~4,~Sec.~7]{Rdiff}). Here we denote the weak-$\lebe^p$ spaces by $\lebe^p_\mathrm{w}$. In particular, $D^{n-j}u\in\lebe^1_{\locc}$ for $j=1,\ldots,n$, so that $u\in\bv^\B_{\locc}$. We then have that
	\begin{align}\label{eq:u_bad}
	u=H_0w+\log|\cdot|\mathbf{L}w,
	\end{align}
	where, recall, $H_0$ is $0$-homogeneous and smooth away from zero and $\mathbf{L}$ is a linear map that depends on $\B$ only. In particular, since $u$ is locally bounded, we have that $\mathbf{L}w=0$. In this case, we have that $u=H_0w$ is $0$-homogeneous and continuous (at zero). It follows that $u$ is constant, so that $\delta_0w=\B u=0$. We can conclude by Lemma~\ref{lem:char(w)canc}.
\end{proof}

\begin{lemma}\label{lem:nec_strict}
	Let $\B$ be an elliptic and weakly canceling operator of order $n$ on $\R^n$. Suppose that the inclusion $\bv^\B_{\locc}(\R^n)\subset\lebe_{\locc}^\infty(\R^n,V)$ is $\B$-strictly continuous. Then $\B$ is canceling.
\end{lemma}
\begin{proof}
	We let $u\in\bv^\B_{\locc}$ be as in the proof of Proposition~\ref{prop:nec_canc}, hence as in \eqref{eq:u_bad}. By weak cancellation, we have that $u=H_0w$. By taking a regularization of $u$ with smooth, compactly supported kernels, we find a sequence of smooth functions that converges to $u$ $\B$-strictly. By the assumption of the lemma, we have that the regularization converges to $u$ locally uniformly, so $u$ is continuous. We conclude as in the proof of Proposition~\ref{prop:nec_canc}.
\end{proof}
\begin{proposition}\label{prop:EC_subset_C}
	Let $\B$ be an $n$-th order operator on $\R^n$ that is elliptic and canceling and $\Omega\subset\R^n$ be an open set. Then $\bv^\B(\Omega)\subset\hold(\Omega,V)$.
\end{proposition}
\begin{proof}
	Let $u\in\bv^\B(\Omega)$ and $x\in\Omega$. We will consider radii $r>0$ such that $B_{2r}(x)\subset\Omega$ and smooth cut-off functions $\rho_r\in\hold^\infty_c(B_r(x))$ such that $|D^j\rho_r|\leq cr^{-j}$ for $j=0,\ldots, n$. By the strict density Lemma~\ref{lem:BVB_density}, we have that the inequality \eqref{eq:Linfty} holds for $\rho_r(u-(u)_{x,r})\in\bv^\B(\R^n)$ (to carefully check this, one can employ \cite[Lem.~2.2]{Rdiff}). Here $(u)_{x,r}=\fint_{B_r(x)}u(y)\dif y$. We estimate:
	\begin{align*}
	\|u-(u)_{x,r}\|_{\lebe^\infty(B_r(x))}&\leq \|\rho_r(u-(u)_{x,r})\|_{\lebe^\infty(B_{2r}(x))}\leq c|\B \left[\rho_r(u-(u)_{x,r})\right]|(B_{2r}(x))\\
	&\leq c\left(|\B u|(B_{2r}(x))+\sum_{j=1}^{n-1}r^{-j}\|D^{n-j}u\|_{\lebe^1(B_{2r}(x))}\right.\\
	&\Bigg.+r^{-n}\|u-(u)_{x,r}\|_{\lebe^1(B_{2r}(x))}\Bigg)\\
	&\leq c\left(|\B u|(B_{2r}(x))+\sum_{j=1}^{n-1}r^{-j}\|D^{n-j}u\|_{\lebe^1(B_{2r}(x))}\right)\\
	&\leq c\left(|\B u|(B_{2r}(x))+\sum_{j=1}^{n-1}\|D^{n-j}u\|_{\lebe^{\frac{n}{n-j}}(B_{2r}(x))}\right),
	\end{align*}
	where in the third inequality we used the Leibniz rule and triangle inequality, the fourth estimate follows from Poincar\'e's inequality (see, e.g., \cite[Lem.~4.1.3]{Ziemer}), and the last estimate follows from H\"older's inequality.
	
	We next note that by the dominated convergence theorem, we have that \\ $|\B u|(B_{2r}(x))\rightarrow |\B u|(\{x_0\})$ as $r\downarrow 0$, which is null, e.g., by \cite{ARDPHR}. By \eqref{eq:VS_j} and the density lemma, we have that $D^{n-j}u\in\lebe^{n/(n-j)}_{\locc}(\Omega)$, so that the summands also tend to zero by dominated convergence theorem. In particular, we showed that
	\begin{align}\label{eq:main}
	\|u-(u)_{x,r}\|_{\lebe^\infty(B_r(x))}\leq \omega(r;x)
	\end{align}
	for some increasing $\omega(\,\cdot\,,x)\colon[0,\infty)\rightarrow[0,\infty]$ with $\lim_{r\downarrow0}\omega(r;x)=0$.
	
	We next show that for each $x$, $((u)_{x,r})_{r>0}$ converges. We know from the Lebesgue differentiation theorem that this is the case $\mathscr{L}^n$-almost everywhere, so we redefine $u$ by
	\begin{align}\label{eq:precise_rep}
	u^*(x)=\lim_{r\downarrow0} (u)_{x,r}
	\end{align}
	at all Lebesgue points $x$. By the triangle inequality, we have that
	\begin{align*}
	\|u-u^*(x)\|_{\lebe^\infty(B_r(x))}\leq \omega(r;x)+|(u)_{x,r}-u^*(x)|\rightarrow0\text{ as }r\downarrow0.
	\end{align*} 
	We assume that there exists a point $x_0\in\R^n$ which is not a Lebesgue point. We will show that $((u)_{x_0,r_m})_{m}$ is Cauchy, where $r_m=m^{-1}$. To do this, let $m\leq k$ and consider Lebesgue points $y_k\in B_{r_k}(x_0)$, so that \eqref{eq:main} implies
	\begin{align*}
	|(u)_{x_0,r_m}-(u)_{x_0,r_k}|&\leq |(u)_{x_0,r_m}-u(y_k)|+|u(y_k)-(u)_{x_0,r_k}|\\
	&\leq \omega(r_m;x_0)+\omega(r_k;x_0),
	\end{align*}
	which converges to zero as $m,\,k\rightarrow\infty$. Write $u^*(x_0)\coloneqq \lim_{m\rightarrow\infty}(u)_{x_0,r_m}$. The same estimation gives
	\begin{align*}
	|(u)_{x_0,r}-(u)_{x_0,r_m}|\leq 2\omega(x_0;r)\quad\text{ if }0<r_m<r,
	\end{align*}
	so that letting $m\rightarrow\infty$, we obtain that  $\lim_{r\downarrow0}(u)_{x_0,r}=u^*(x_0)$. 
	For Lebesgue points $x\in B_r(x_0)$ of $u$, we then have that
	\begin{align*}
	|u(x)-u^*(x_0)|\leq |u(x)-(u)_{x_0,r}|+|(u)_{x_0,r}-u^*(x_0)|,
	\end{align*}
	which proves that
	\begin{align}\label{eq:modulus}
	\|u-u^*(x_0)\|_{\lebe^\infty(B_r(x_0))}\leq 3\omega(r;x_0)
	\end{align}
	which is a contradiction. Hence all $x$ are Lebesgue points of $u$, and, moreover, $u$ (identified with $u^*$, hence defined \emph{pointwise everywhere} by \eqref{eq:precise_rep}) is $\lebe^\infty$-continuous everywhere. It is not difficult to see from the proof above that the inequality \eqref{eq:modulus} holds at every $x_0$, hence clearly implies continuity of $u$.
\end{proof}
The strict continuity result of Theorem~\ref{thm:main} will follow as a consequence of the following:
\begin{lemma}\label{lem:my_first_MT_lemma}
	Let $Q\subset\R^n$ be a cube and $\mu_j\in\mathcal{M}(\bar Q)$ be a sequence of positive measures such that $\mu_j\wstar \mu$ in $\mathcal{M}(\bar Q)$, where $\mu_j,\,\mu$ are non-atomic. Then for each $\varepsilon>0$ there exists $\delta>0$ such that for each sub-cube $\tilde Q\subset Q$ with $|\tilde Q|<\delta$, we have that $\sup_j \mu_j(\tilde Q)<\varepsilon$ (all cubes considered are closed and have faces paralled to the coordinate hyperplanes).
\end{lemma}
\begin{proof}
	Assume for contradiction that there exist $\varepsilon>0$, a sub-sequence $\mu_{j_i}$, and a sequence of cubes $\bar Q(x_i,r_i)$ such that $r_i\downarrow 0$ as $i\rightarrow\infty$ and $\mu_{j_i}(\bar Q(x_i,r_i))\geq \varepsilon$. If $(j_i)_i$ has a subsequence increasing to infinity, then we can, by passing to another subsequence (neither relabelled), assume that $x_i\rightarrow x\in\bar Q$. For given (large) $l$ and sufficiently large $i$, we have that $\bar{Q}(x_i,r_i)\subset \bar{Q}(x,l^{-1})$, so that $\mu_{j_i}(\bar{Q}(x,l^{-1}))\geq\varepsilon$. By \cite[Prop.~1.203(ii)]{FL_book}, we have that $\mu(\bar{Q}(x,l^{-1}))\geq \varepsilon$. Letting $l\rightarrow\infty$, we have by the dominated convergence theorem that $\mu(\{x\})\geq \varepsilon>0$, which leads to a contradiction.
	
	If, on the other hand, we have that $(j_i)_i$ has no sub-sequence increasing to infinity, then it is bounded, and, in particular, must have a stationary subsequence, not-relabelled, $j_i=J$. Then we can repeat the argument above after replacing $\mu_{j_i}=\mu_J=\mu$. The proof is complete.
\end{proof}
\begin{proposition}\label{prop:strict}
	Let $\B$ be an $n$-th order elliptic and canceling operator on $\R^n$. If $u_j,\,u\in\bv^\B(\R^n)$ are such that $u_j$ converge to $u$ $\B$-strictly, then $u_j$ converge to $u$ uniformly.
\end{proposition}
\begin{proof}
	Let $\varepsilon > 0$. By Lemmas~\ref{lem:WB1_ctnity},~\ref{lem:BVB_density}, and \ref{lem:prokhorov}, arguing as in the proof of Lemma~\ref{lem:BVB_strict}, we have tightness of $\{|\B (u_j - u)|\}_j$ and $\{|D^{n-l}(u_j - u)|\}_j$, $l=1,\ldots, n$. Thus we can choose a sufficiently large cube $Q$ such that the mass of both of these is uniformly (in $j$) less than $\varepsilon$ outside $\bar Q$.
	We aim to use the Ascoli-Arzela Theorem in $C(\overline{2Q},V)$, where $2Q$ is a concentric cube to $Q$, of twice the side length. For this purpose, we note that $u_j$ are uniformly bounded by \eqref{eq:Linfty}. 
	
	We next argue that the family $\{u_j\}_j$ is equi-continuous in $\overline{2Q}$. Let $\eta>0$. Take a smooth cut-off function $\rho\in\hold^\infty_c(4Q)$ such that $\rho=1$ in $\overline{2Q}$. It is then elementary to show by use of the product rule that $\rho u_j \rightarrow \rho u$ $\B$-strictly. We have that $u_j$ are continuous by Proposition~\ref{prop:EC_subset_C}, so we can infer from \eqref{eq:main} that
	\begin{align}\label{eq:eqc}
	|u_j(x)-u_j(y)|&\leq c\left(|\B u_j|(Q_{2r})+\sum_{l=1}^{n-1}\|D^{n-l} u_j\|_{\lebe^{\frac{n}{n-l}}(Q_{2r})}\right)
	\end{align}
	whenever $x,\,y$ lie in the same cube $Q_r$ of side-length $r$ such that $Q_{2r}\Subset 2Q$ ($Q_r$ and $Q_{2r}$ are assumed concentric). By Proposition~\ref{lem:BVB_strict} and Vitali's Theorem we have that $\{D^{n-l}(\rho u_j)\}_j$ are $\frac{n}{n-l}$-uniformly integrable in $4Q$, and thus $\{D^{n-l}u_j\}_j$ are $\frac{n}{n-l}$-uniformly integrable in $2Q$. Hence there exists $\delta>0$ such that if $0<r<\delta$, we have that the sum on the right hand side of \eqref{eq:eqc} is less than $\eta$. By Lemma~\ref{lem:my_first_MT_lemma}, possibly by making $\delta$ smaller, we have that $|\B (\rho u_j)|(\tilde{Q}_{2r})\leq \eta$ for every $\tilde{Q}_{2r} \subset 4Q$ independently of $j$ and of the choice of $r<\delta$. In particular, this is true if we take $\tilde{Q}_{2r}$ to be our arbitrary cube $Q_{2r} \Subset 2Q$, in which case $|\B (\rho u_j)|(Q_{2r}) = |\B u_j|(Q_{2r})$ by our choice of $\rho$.
	
	It follows that $\{u_j\restriction_{\overline{2Q}}\}_j$ is pre-compact in $\hold(\overline{2Q},V)$. Then there is a subsequence $\{u_{j_i}\restriction_{\overline{2Q}}\}_i$ which converges uniformly, say, to $\tilde u$, hence $u_{j_i}\restriction_{\overline{2Q}}\rightarrow \tilde u$ in $\lebe^1(\overline{2Q})$, {so $\tilde u=u\restriction_{\overline{2Q}}$}. Now for $\eta\in\hold^\infty(\R^n)$ such that $\eta=1$ in $\R^n \backslash 2Q$ and $\eta = 0$ in $\bar Q$, satisfying $\|D^{l}u\|_{\lebe^\infty}\leq c$ for $l=1,\ldots,n-1$ we have
	\begin{align*}
	    \| u_{j_i} - u \|_{\lebe^\infty} &\leq \| u_{j_i} - u \|_{\lebe^\infty(2Q)} + \| u_{j_i} - u \|_{\lebe^\infty(\R^n \backslash 2Q)} \\
	    &\leq \| u_{j_i} - u \|_{\lebe^\infty(2Q)} + \| \eta(u_{j_i} - u) \|_{\lebe^\infty(\R^n \backslash \bar Q)} \\
	    &\leq c\left(\| u_{j_i} - u \|_{\lebe^\infty(2Q)} + | \B[\eta(u_{j_i} - u)]|(\R^n \backslash \bar Q)\right) \\
	    &\leq c\left(\| u_{j_i} - u \|_{\lebe^\infty(2Q)} + | \B(u_{j_i} - u)|(\R^n \backslash \bar Q) +\right. \\
	    &\left.+\sum_{l=1}^{n-1}\|D^{n-l} (u_{j_i}- u)\|_{\lebe^{1}(\R^n \backslash \bar Q)}\right) \\
	    &\leq c\varepsilon
	\end{align*}
	for $i$ sufficiently large. 
	
	It follows that the sequence $\{u_j\}_j$ has a unique cluster point with respect to the strict convergence in $\bv^\B$, which is complete when equipped with the metric of strict topology. It follows that $u_j\rightarrow u$ uniformly, which concludes the proof.
\end{proof}

	\section{$\C$-elliptic operators}\label{sec:C-ell}
	We are now considering the question of continuity up to the boundary of $\bv^\B$-maps on bounded domains, for which we choose the prototypical example of a cube. 
	The reader might be tempted to think that the claim of Theorem~\ref{thm:main2} follows from Theorem~\ref{thm:main} via an extension theorem in the spirit of \cite[Thm.~1.2]{GR}. However, so far, such an extension theorem for higher order operators is only known for $\sobo^{\B,1}(Q)$ and \emph{not} for $\bv^\B(Q)$. On the other hand, our method explicitly gives a simple and quite hands on estimate on the modulus of continuity up to the boundary, as we will present in the following:
	\begin{proposition}\label{prop:suff_C-ell}
		Let $\B$ be a $\C$-elliptic $n$-th order operator on $\R^n$ and $Q\subset\R^n$ be an open cube. Then $\bv^\B(Q)\subset\hold(\bar Q,V)$.
	\end{proposition}
	\begin{proof}
		Let $u\in\bv^\B(Q)$. By Proposition~\ref{prop:EC_subset_C} and \cite[Lem.~3.2]{GR} we have that $u\in\hold(\Omega,V)$. It remains to check continuity up to the boundary. To this end, let $x_0\in\partial\Omega$. For small enough radii $r>0$, we note that  $Q(x_0,r)\cap\Omega$ is one of finitely many rectangles with proportional side lengths (fixed up to homothety). We will abuse notation and denote $Q(x_0,r)\cap\Omega$ also by $Q(x_0,r)$. One can then use Proposition~\ref{prop:GR} and a scaling argument to prove that:
		\begin{align*}
		\|u-(u)_{x,r}\|_{\lebe^\infty(Q_r(x_0))}&\leq c\left(|\B u|(Q_{r}(x_0))+{r^{-n}}\|u-(u)_{x,r}\|_{\lebe^1(Q_r(x_0))}\right)\\
		&\leq c\left(|\B u|(Q_{r}(x_0))+{r^{1-n}}\|D u\|_{\lebe^1(Q_{r}(x_0))}\right)\\
		&\leq c\left(|\B u|(Q_{r}(x_0))+\|D u\|_{\lebe^n(Q_{r}(x_0))}\right),
		\end{align*}
		where in the second estimate we used Poincar\'e's inequality and in the third estimate we used H\"older's inequality. We conclude as in the proof of Proposition~\ref{prop:EC_subset_C}, using the fact that $\bv^\B(\Omega)\subset\sobo^{1,n}(\Omega,V)$ for $\C$-elliptic operators, which follows from \cite{GR}, arguing similarly to the proof of Proposition~\ref{prop:GR}.
	\end{proof}
	\begin{remark}
	{In particular, the estimate for the modulus of continuity thus obtained is
	\begin{align*}
		|u(x)-u(y)|\leq c\left(|\B u|(Q_{r})+\|D u\|_{\lebe^n(Q_{r})}\right)\quad\text{ for }x,\,y\in Q_r\subset \bar Q,
	\end{align*}
	where $Q_r$ denotes a cube of radius $r$ with faces parallel to the coordinate axes. This estimate is a generalization of the inequality in \cite[Eq.~(1.2)]{PVS}, which was the starting point for this work. In the case of elliptic and canceling operators, we have the weaker estimate \eqref{eq:eqc}, in the sense that more terms are needed to control the oscillations and the support of the estimation is increased.}
	\end{remark} 
	However, one cannot expect to keep the support of estimation fixed in the latter case because, otherwise, one can obtain boundary estimates for elliptic and canceling operators by the proof of Proposition~\ref{prop:suff_C-ell}. This is not possible by \cite[Counterex.~3.4]{GR} and the following:
	\begin{lemma}\label{lem:nec_C-ell}
		Let $\B$ be an $n$-th order operator on $\R^n$ and $\Omega\subset\R^n$ be a bounded Lipschitz domain. Suppose that $\bv^\B(\Omega)\subset\hold(\bar\Omega,V)$. Then $\B$ is $\C$-elliptic.
	\end{lemma}
	\begin{proof}
		We only prove the claim under the assumption that $\B$ is elliptic. Necessity of ellipticity follow by a simplification of the arguments to follow.
		
		If $\B$ is elliptic but not $\C$-elliptic, there exist non-zero complex vectors $\xi,\,v$ such that $\B(\xi)v=0$. It is shown in \cite[Prop.~3.1]{GR} that for the complex valued maps $u(x)=f((x-x_0)\cdot\xi)v$, $x\in\Omega$, we have that $\B u=0$ at all points where $f$ is holomorphic. 
		Moreover, it was shown that $D^lu(x)=(\partial^l_1f)(x\cdot\xi)v\otimes \xi^{\otimes l}$. After a change of variable, we can assume that $D^lu(x)=(\partial^l_1f)(x_1+\imag x_2)V_l$, where $V_l=v\otimes \xi^{\otimes l}$. In particular, if $Q=[-R,R]^n$,
		\begin{align}\label{eq:integral}
			\int_Q |D^lu(x)|\dif x=|V_l|\int_{[-R,R]^2}|(\partial^l_1f)(x_1+\imag x_2)|\dif\mathscr{L}^2(x_1,x_2).
		\end{align}
		
		We choose $x_0\in\partial\Omega$ such that the $(n-1)$-dimensional half-space $x_0+\{x\in\R^n\colon x_2\leq 0\}$ does not intersect $\Omega$ in a small ball $B_{r_0}(x_0)$. We choose $f$ such that $\partial^{n-1}f(z)=z^{-1}$ in $\C\setminus\imag(-\infty,0]$. Such $f$ exists by standard results of complex analysis. It then follows from \eqref{eq:integral} that $D^{n-1}u$ is integrable in $ B_{r_0}(x_0)\cap\Omega$. Using \cite[Sec.~1.1.11]{Mazya}, we obtain that $u\in\sobo^{n-1,1}(\Omega\cap B_{r_0}(x_0),V)$. Taking a cut-off $\rho\in\hold^\infty_c(B_{r_0}(x_0))$ such that $\rho=1$ in $B_{r_0/2}(x_0)$, we conclude that by use of the product rule that $\tilde{u}\coloneqq \rho u\in\bv^\B(\Omega)$, but $\tilde u$ is not continuous at $x_0$.
	\end{proof}
	We can now conclude the proof of Theorem~\ref{thm:main2}, and with it, the present paper.
	\begin{proof}[Proof of Theorem~\ref{thm:main2}]
		We already proved in Proposition~\ref{prop:suff_C-ell} and Lemma~\ref{lem:nec_C-ell} that $\C$-ellipticity of $\B$ is equivalent with $\bv^\B(Q)\subset\hold(\bar Q,V)$. In particular, the restriction map $\tr u=u\restriction_{\partial Q}$ is a well-defined linear trace operator on $\bv^\B(Q)$. To see that $\tr$ is bounded, we use Proposition~\ref{prop:GR} and estimate
		\begin{align*}
			\|\tr u\|_{\lebe^\infty(\partial Q)}\leq \| u\|_{\lebe^\infty(Q)}\leq c\|u\|_{\bv^\B(\Omega)}.
		\end{align*}
		The proof is complete.
	\end{proof}

\end{document}